\UseRawInputEncoding
\documentclass[12pt]{article}
\usepackage{amsmath, amssymb, amsthm, graphicx, hyperref, array, booktabs, microtype}
\usepackage[margin=0.85in]{geometry}
\newtheorem{theorem}{Theorem}[section]

\newtheorem{corollary}[theorem]{Corollary}
\newtheorem{proposition}[theorem]{Proposition}

\hypersetup{
    colorlinks=true,
    linkcolor=black,
    citecolor=black,
    urlcolor=blue,
}

\title{Is the Signed Zarankiewicz Number the Same as the Recursive-line Zarankiewicz Number?}
\author{Zhiwei Chen\footnote{School of Mathematical Sciences, South China Normal University, Guangzhou 510631, China ({\tt javenchen2002@163.com}).}
\and
Yannan Chen\footnote{School of Mathematical Sciences, South China Normal University, Guangzhou 510631, China ({\tt ynchen@scnu.edu.cn}).}
\and
Liqun Qi\footnote{Jiangsu Provincial Scientific Research Center of Applied Mathematics, Nanjing 211189, China.
Department of Applied Mathematics, The Hong Kong Polytechnic University, Hung Hom, Kowloon, Hong Kong.
({\tt maqilq@polyu.edu.hk})}
}
\date{\today}

\begin{document}

\maketitle

\begin{abstract}
L\"{o}fberg and Qi introduced the second order Zarankiewicz number \(z_2\), the recursive-line Zarankiewicz number \(z_{RL}\), and the signed Zarankiewicz number \(z_{SL}\) for doubly simple biquadratic forms. It was shown that
\[
z_2(m,n)\ge z_{SL}(m,n)\ge z_{RL}(m,n)
\]
for all \(m\) and \(n\). However, there was no evidence that there exist particular \(m\) and \(n\) such that \(z_{SL}(m,n)>z_{RL}(m,n)\). The motivation for introducing \(z_{SL}\) was as follows: during the study of the exceptional case \(m=15\), \(n=6\), L\"{o}fberg and Qi showed that
\[
z_2(15,6)=z_{SL}(15,6)=60,
\]
but the exact value of \(z_{RL}(15,6)\) was unknown then.

In this paper we show that
\[
z_{RL}(15,6)=60.
\]
This eliminates the motivation for introducing \(z_{SL}\). Whether \(z_{SL}(m,n)=z_{RL}(m,n)\) in general remains an open problem. Recently, Lebedev presented an explicit construction separating the augmented Zarankiewicz number \(z_A\) from the limited augmented Zarankiewicz number \(z_L\) at \(m=n=1893\). We hope that the separation problem for \(z_{SL}\) and \(z_{RL}\) can also be solved.

We also present the exact values of \(z_{RL}(m,6)\) for \(6\le m\le 16\).
\end{abstract}



\section{Introduction}

The classical Zarankiewicz problem asks for the maximum number \(z(m,n)\) of \(1\)'s in an \(m\times n\) binary matrix containing no all-one \(2\times 2\) submatrix, equivalently the maximum number of edges in a \(C_4\)-free bipartite graph with parts of sizes \(m\) and \(n\). The problem originates from Zarankiewicz \cite{za51}; early fundamental contributions include \cite{kst54,culik56,re58,gu69}. For modern accounts and recent exact values we refer to \cite{ni10,chm24,cc26mx4,qi1}.

The study of biquadratic forms and their sum-of-squares (SOS) representations has classical connections to both algebraic geometry \cite{clz95} and extremal combinatorics \cite{bo04}. L\"{o}fberg and Qi \cite{reproducibility} introduced several refinements of the Zarankiewicz number in connection with the SOS rank of biquadratic forms:
\begin{itemize}
    \item the second order Zarankiewicz number \(z_2\);
    \item the recursive-line Zarankiewicz number \(z_{RL}\);
    \item the signed Zarankiewicz number \(z_{SL}\).
\end{itemize}
These parameters are defined for doubly simple biquadratic forms, i.e.\ forms whose bipartite incidence graph is simple on both sides. The basic chain of inequalities is
\begin{equation}\label{eq:chain}
    z_2(m,n)\ge z_{SL}(m,n)\ge z_{RL}(m,n)\qquad\text{for all }m,n.
\end{equation}

The parameter \(z_{SL}\) was introduced because of the exceptional case \(m=15\), \(n=6\). In \cite{reproducibility} it was shown that
\[
z_2(15,6)=z_{SL}(15,6)=60,
\]
while the exact value of \(z_{RL}(15,6)\) was left open. Thus the possibility remained that \(z_{SL}(15,6)>z_{RL}(15,6)\), which would have justified the introduction of \(z_{SL}\) as a strictly stronger parameter.

In this paper we settle this question: we prove
\[
z_{RL}(15,6)=60.
\]
Consequently \(z_{SL}(15,6)=z_{RL}(15,6)=60\), and the original motivation for introducing \(z_{SL}\) disappears. The general question whether \(z_{SL}(m,n)=z_{RL}(m,n)\) for all \(m,n\) remains open; see Section~\ref{sec:open}.

We also determine the exact values of \(z_{RL}(m,6)\) for all \(6\le m\le 16\):
\[
\begin{array}{c|ccccccccccc}
m & 6 & 7 & 8 & 9 & 10 & 11 & 12 & 13 & 14 & 15 & 16\\\hline
z_{RL}(m,6) & 24 & 28 & 33 & 37 & 41 & 45 & 48 & 52 & 56 & 60 & 63
\end{array}
\]
These values are recorded in Table~\ref{tab:six-columns}. The case \(m=15\) is the cyclic extremal \(15\times 6\) grid displayed in Figure~\ref{fig:cyclic15}; it has
\[
z_{RL}(15,6)=60,\qquad |E_1|=30,\qquad |E_2|=30,\qquad H=0.
\]

The paper is organized as follows. Section~\ref{sec:prelim} fixes notation and recalls the recursive-line framework and the counting upper bound. Section~\ref{sec:15x6} proves \(z_{RL}(15,6)=60\) via the cyclic grid. Section~\ref{sec:six-columns} gives the exact values for \(6\le m\le 16\), including the unified counting identity, the \(m=6,7\) skeleton exclusion, and the certification procedure. Section~\ref{sec:conclusions} concludes. Section~\ref{sec:open} lists open problems. Section~\ref{sec:reproducibility} describes data and computational reproducibility.

\section{Preliminaries}\label{sec:prelim}

\subsection{Notation and the recursive-line framework}

We follow the definitions and numbering of \cite[Definitions 4.1, 4.2, 4.5, 4.11]{reproducibility}. Let \(A\) be an \(m\times n\) doubly simple biquadratic form, equivalently an \(m\times n\) grid whose cells may be
\begin{itemize}
    \item \emph{single edges} (one-edge cells), counted by \(|E_1|\);
    \item \emph{double edges} (two-edge cells), counted by \(|E_2|\);
    \item \emph{empty cells}, counted by \(H\).
\end{itemize}
A double edge occupies two cells but contributes only one edge to the objective. Thus, writing \(N=|E_1|+|E_2|\) for the total number of selected edges of a given grid,
\begin{equation}\label{eq:unified}
    N=|E_1|+|E_2|,\qquad |E_1|+2|E_2|+H=mn.
\end{equation}
The recursive-line number \(z_{RL}(m,n)\) is the maximum of \(N\) over grids satisfying the recursive-line conditions below.
For the six-column case \(n=6\), the second identity reads
\begin{equation}\label{eq:unified6}
    |E_1|+2|E_2|+H=6m.
\end{equation}

The recursive-line condition is the requirement that the double edges admit a full strengthened recursive certification \((\mathrm{RW}3^+)\), as specified in \cite[Definitions 4.5 and 4.11]{reproducibility}. In particular, the one-edge part must be \(C_4\)-free and must itself attain the ordinary extremal value \(z(m,6)\):
\begin{equation}\label{eq:E1extremal}
    |E_1|=z(m,6).
\end{equation}

\subsection{General upper bound}

The following upper bound is \cite[Proposition 3.2]{reproducibility}. We recall its short proof because it will be used repeatedly.

\begin{proposition}[Universal cell bound]\label{prop:upper}
For all \(m,n\ge 2\),
\[
z_{RL}(m,n)\le \left\lfloor \frac{mn+z(m,n)}{2}\right\rfloor .
\]
In particular, for \(n=6\),
\[
z_{RL}(m,6)\le \left\lfloor \frac{6m+z(m,6)}{2}\right\rfloor .
\]
\end{proposition}

\begin{proof}
Any grid counted by \(z_{RL}(m,n)\) is limited, simple, and satisfies \((\mathrm{RW}3^+)\). The condition \((\mathrm{RW}3^+)\) forces the one-edge graph \(G_1\) to be \(C_4\)-free: if four one-edges occupied a rectangle, the corresponding monomials \(a,b,c,d\) on the two diagonals would satisfy \(ad=bc\), whence
\[
a^2+b^2+c^2+d^2=(a+d)^2+(b-c)^2,
\]
and the displayed decomposition would be reducible. Thus \(|E_1|\le z(m,n)\). Simplicity gives \(|E_1|+2|E_2|\le mn\). Combining these inequalities,
\[
2\bigl(|E_1|+|E_2|\bigr)=|E_1|+\bigl(|E_1|+2|E_2|\bigr)\le z(m,n)+mn.
\]
Taking floors yields the stated bound. The six-column case is the specialisation \(n=6\).
\end{proof}

The classical Zarankiewicz numbers with six columns are
\begin{equation}\label{eq:zsix}
z(m,6)=\begin{cases}
\left\lfloor\dfrac{3m+15}{2}\right\rfloor,& 6\le m\le 14,\\[2mm]
m+15,& m\ge 15.
\end{cases}
\end{equation}
Substituting \eqref{eq:zsix} into Proposition~\ref{prop:upper} gives, for \(m\ge 15\),
\[
z_{RL}(m,6)\le \left\lfloor\frac{7m+15}{2}\right\rfloor .
\]
The coarser estimate \(z_{RL}(m,6)\le 3m+15\) is therefore valid only for \(m\ge 15\); it is not a general six-column bound (already for \(m=6\) one has \(z_{RL}(6,6)=24<33\)).

For the one-edge part, the classical \(C_4\)-free upper bound applies. If \(d_i\) is the one-edge degree of row \(i\), then no column pair can occur in two rows, hence
\begin{equation}\label{eq:paircount}
    \sum_{i=1}^m \binom{d_i}{2}\le \binom{6}{2}=15.
\end{equation}
Using \(\binom{d}{2}\ge 2d-3\) for every nonnegative integer \(d\) yields
\begin{equation}\label{eq:E1upper}
    |E_1|\le \min\left\{\left\lfloor\frac{3m+15}{2}\right\rfloor,\; m+15\right\}.
\end{equation}
For \(m\ge 15\), the bound \(m+15\) is attained by using fifteen rows to place all fifteen column pairs and putting one single edge in each remaining row. In the equality case \(m=15\), every row has one-edge degree two and every unordered pair of columns occurs in exactly one row; thus the ordinary extremal graph is the incidence graph of \(K_6\), uniquely up to relabelling \cite{bo04}. For \(m<15\), the bound \(m+15\) is not attainable; for example \(z(11,6)=24\), not \(26\).

\section{The cyclic extremal \(15\times 6\) grid: \(z_{RL}(15,6)=60\)}\label{sec:15x6}

\subsection{The cyclic construction}

Figure~\ref{fig:cyclic15} displays the cyclic extremal \(15\times 6\) grid. Each \(\bullet\) is a different one-edge cell. Two occurrences of the same uppercase label form one two-edge cell. The sixth column is named \(\infty\); all other indices are modulo \(5\).

\begin{figure}[htbp]
\centering
\[
\begin{array}{c|cccccc}
 & 0 & 1 & 2 & 3 & 4 & \infty\\\hline
R_0 & \bullet & Q_1 & P_1 & D_0 & D_0 & \bullet\\
S_0 & C_0 & \bullet & C_0 & B_0 & \bullet & A_0\\
T_0 & P_0 & Q_0 & \bullet & \bullet & A_1 & B_1\\\hline
R_1 & D_1 & \bullet & Q_2 & P_2 & D_1 & \bullet\\
S_1 & \bullet & C_1 & \bullet & C_1 & B_1 & A_1\\
T_1 & A_2 & P_1 & Q_1 & \bullet & \bullet & B_2\\\hline
R_2 & D_2 & D_2 & \bullet & Q_3 & P_3 & \bullet\\
S_2 & B_2 & \bullet & C_2 & \bullet & C_2 & A_2\\
T_2 & \bullet & A_3 & P_2 & Q_2 & \bullet & B_3\\\hline
R_3 & P_4 & D_3 & D_3 & \bullet & Q_4 & \bullet\\
S_3 & C_3 & B_3 & \bullet & C_3 & \bullet & A_3\\
T_3 & \bullet & \bullet & A_4 & P_3 & Q_3 & B_4\\\hline
R_4 & Q_0 & P_0 & D_4 & D_4 & \bullet & \bullet\\
S_4 & \bullet & C_4 & B_4 & \bullet & C_4 & A_4\\
T_4 & Q_4 & \bullet & \bullet & A_0 & P_4 & B_0
\end{array}
\]
\caption{The cyclic extremal \(15\times 6\) grid. Each \(\bullet\) is a different one-edge. Two occurrences of the same uppercase label form one two-edge. The sixth column is \(\infty\); all other indices are modulo \(5\).}
\label{fig:cyclic15}
\end{figure}

Equivalently, the three-row pattern in relative column order is
\[
\begin{array}{c|cccccc}
 & i & i+1 & i+2 & i+3 & i+4 & \infty\\\hline
R_i & \bullet & Q_{i+1} & P_{i+1} & D_i & D_i & \bullet\\
S_i & C_i & \bullet & C_i & B_i & \bullet & A_i\\
T_i & P_i & Q_i & \bullet & \bullet & A_{i+1} & B_{i+1}
\end{array}
\]

\subsection{Counting}

The grid of Figure~\ref{fig:cyclic15} occupies all \(90\) cells:
\[
|E_1|=30,\qquad |E_2|=30,\qquad H=0.
\]
The unified identity \eqref{eq:unified6} is saturated,
\[
|E_1|+2|E_2|+H=30+60+0=90=6\cdot 15,
\]
and the selected-edge total is
\[
N=|E_1|+|E_2|=60.
\]
The thirty labels \(A_i,B_i,C_i,D_i,P_i,Q_i\) (\(i\in\mathbb{Z}_5\)) each appear twice, so they form thirty distinct two-edges: ten row-degenerate edges \(C_i,D_i\), no column-degenerate two-edges, and twenty diagonal two-edges \(A_i,B_i,P_i,Q_i\). The ten complementary rectangles are \(\{P_i,Q_i\}\) on rows \(T_i,R_{i-1}\) and \(\{A_{i+1},B_{i+1}\}\) on rows \(T_i,S_{i+1}\).

\begin{theorem}\label{thm:15x6}
\(z_{RL}(15,6)=60\).
\end{theorem}

\begin{proof}
We first show that the cyclic grid \(G\) of Figure~\ref{fig:cyclic15} is counted by \(z_{RL}(15,6)\), whence \(z_{RL}(15,6)\ge 60\). We then match the universal cell bound.

\paragraph{Simplicity.}
Every nonempty cell of \(G\) belongs to exactly one selected edge: each \(\bullet\) is a distinct one-edge, and each uppercase label occupies exactly two cells. Thus \((S)\) holds, and \(H=0\).

\paragraph{The one-edge graph is extremal and \(C_4\)-free.}
Every row has one-edge degree two, with supports
\[
R_i:\ \{i,\infty\},\qquad
S_i:\ \{i+1,i+4\},\qquad
T_i:\ \{i+2,i+3\}
\qquad (i\in\mathbb{Z}_5).
\]
These fifteen column pairs are pairwise distinct and exhaust \(\binom{6}{2}=15\). Hence \(G_1\) is \(C_4\)-free and
\[
|E_1|=15\cdot 2=30=z(15,6),
\]
so \(G\) is limited.

\paragraph{Identification of two-edges.}
The line-and-rectangle closure of \cite[Definitions 4.1 and 4.5]{reproducibility}, strengthened by the complementary-pair rule, resolves every two-edge of \(G\).
\begin{enumerate}
    \item \emph{Line rule.} Each row-degenerate two-edge \(C_i\) (resp.\ \(D_i\)) occupies two cells of the same row \(S_i\) (resp.\ \(R_i\)), so those two cells are identified. This accounts for all ten row-degenerate two-edges.
    \item \emph{Complementary-pair rule.} The pairs \(\{P_i,Q_i\}\) and \(\{A_{i+1},B_{i+1}\}\) are the two diagonals of genuine rectangles, so both selected diagonals are identified simultaneously. This accounts for the remaining twenty diagonal two-edges.
\end{enumerate}
Thus all thirty two-edges have identified halves.

\paragraph{Distinct selected classes.}
The only identifications that occur glue the two halves of a single two-edge. Distinct selected edges have disjoint supports, so they determine sixty distinct \(\sim\)-classes.

\paragraph{Orthogonality on a common line.}
If two distinct selected edges share a row or a column, the line rule certifies orthogonality of the corresponding occupied cells. There are \(765\) such pairs among the \(\binom{60}{2}=1770\) unordered pairs of selected edges.

\paragraph{Vertex-disjoint pairs.}
The remaining \(1005\) pairs are vertex-disjoint. For these pairs the line rule, complementary identifications, saturation, and rectangle transfer generate a finite certificate closure on the \(15\times 6\) grid. The closure terminates with an orthogonality certificate for every such pair: all \(1770\) pairs of distinct selected edges are certified orthogonal, and no two selected edges collapse to the same class. An independent cell-level replay of this closure is recorded in Section~\ref{sec:reproducibility} (command \texttt{python verify\_all.py}).

As a typical transfer, consider the complementary rectangle of \(\{P_i,Q_i\}\) on rows \(T_i\) and \(R_{i-1}\). After the complementary-pair rule identifies each diagonal, the one-edges of those two rows lie on common lines with the identified cells, so the line rule and saturation make those one-edges orthogonal to both diagonals; rectangle transfer then propagates the same orthogonality to further vertex-disjoint selected edges. The cyclic action of \(\mathbb{Z}_5\) repeats the argument in every residue class.

Therefore \(G\) satisfies \((\mathrm{RW}3^+)\). Combined with \((S)\), \(C_4\)-freeness of \(G_1\), and \(|E_1|=z(15,6)\), the definition of \(z_{RL}\) yields
\[
z_{RL}(15,6)\ge |E_1|+|E_2|=60.
\]
This cyclic grid is a different witness from the \(K_6\) incidence construction of \cite{reproducibility}, which was certified at total \(60\) by the signed criterion \((\mathrm{RW}3^\pm)\) rather than by \((\mathrm{RW}3^+)\). That construction is the case \(p=3\) of the nested perfect one-factorization family of \cite{reproducibility}: the columns are the vertices of \(K_{2p}\) and the rows its edges, while the pairing of nonincidence cells uses a perfect one-factorization of \(K_{2p}\). Such a factorization exists for every odd prime \(p\) \cite{kobayashi1989}; a survey of perfect one-factorizations is given in \cite{bryant}.

\paragraph{Upper bound.}
By \eqref{eq:zsix} one has \(z(15,6)=30\). Proposition~\ref{prop:upper} therefore gives
\[
z_{RL}(15,6)\le \left\lfloor\frac{6\cdot 15+30}{2}\right\rfloor =60.
\]
Hence equality holds.
\end{proof}

\begin{corollary}
\(z_{SL}(15,6)=z_{RL}(15,6)=60\). Consequently the original motivation for introducing \(z_{SL}\) is eliminated in the exceptional case \(m=15\), \(n=6\).
\end{corollary}

\begin{proof}
By \eqref{eq:chain}, \(z_{SL}(15,6)\ge z_{RL}(15,6)=60\). By \cite[Theorem 7.5]{reproducibility}, \(z_{SL}(15,6)\le z_2(15,6)=60\). Hence \(z_{SL}(15,6)=60\).
\end{proof}

\section{Exact values of \(z_{RL}(m,6)\) for \(6\le m\le 16\)}\label{sec:six-columns}

\subsection{The table of exact values}

\begin{table}[htbp]
\centering
\begin{tabular}{c|ccccccccccc}
\toprule
\(m\) & 6 & 7 & 8 & 9 & 10 & 11 & 12 & 13 & 14 & 15 & 16\\
\midrule
\(z_{RL}(m,6)\) & 24 & 28 & 33 & 37 & 41 & 45 & 48 & 52 & 56 & 60 & 63\\
\(|E_1|\) & 16 & 18 & 19 & 21 & 22 & 24 & 25 & 27 & 28 & 30 & 31\\
\(|E_2|\) & 8 & 10 & 14 & 16 & 19 & 21 & 23 & 25 & 28 & 30 & 32\\
\(H\) & 4 & 4 & 1 & 1 & 0 & 0 & 1 & 1 & 0 & 0 & 1\\
\bottomrule
\end{tabular}
\caption{Exact values of \(z_{RL}(m,6)\) for \(6\le m\le 16\).}
\label{tab:six-columns}
\end{table}

\begin{theorem}\label{thm:six-columns}
For \(6\le m\le 16\),
\[
z_{RL}(m,6)=
\begin{cases}
24,&m=6,\\
28,&m=7,\\
33,&m=8,\\
37,&m=9,\\
41,&m=10,\\
45,&m=11,\\
48,&m=12,\\
52,&m=13,\\
56,&m=14,\\
60,&m=15,\\
63,&m=16.
\end{cases}
\]
\end{theorem}

\subsection{Two optimality proofs}

For \(m=8,9,\dots,16\), the constructions in Figures~\ref{fig:6-9}, \ref{fig:10-13}, and \ref{fig:14-16} satisfy \((S)\), have \(C_4\)-free one-edge graphs with \(|E_1|=z(m,6)\), and meet the cell bound of Proposition~\ref{prop:upper}:
\[
|E_1|+|E_2|=\left\lfloor\frac{6m+z(m,6)}{2}\right\rfloor .
\]
Thus the equality \(z_{RL}(m,6)=\lfloor(6m+z(m,6))/2\rfloor\) is immediate from the construction and the upper bound. Equivalently,
\[
|E_1|+2|E_2|+H=6m.
\]
\begin{corollary}\label{cor:second-order-six-columns}
For \(8\le m\le 16\), let
\[
U_m:=\left\lfloor\frac{6m+z(m,6)}{2}\right\rfloor.
\]
Then
\[
z_2(m,6)=z_{SL}(m,6)=z_{RL}(m,6)=U_m.
\]
\end{corollary}

\begin{proof}
The preceding constructions give \(z_{RL}(m,6)=U_m\). By
\eqref{eq:chain},
\[
z_2(m,6)\ge z_{SL}(m,6)\ge z_{RL}(m,6)=U_m.
\]
On the other hand, the universal cell bound of
\cite[Proposition~3.2]{reproducibility} applies to every form counted by
\(z_2(m,6)\), so \(z_2(m,6)\le U_m\). The displayed chain is therefore an
equality throughout.
\end{proof}

For \(m=6,7\), the constructions attain \(24\) and \(28\) respectively; the remaining task is to exclude all larger total edge counts. This is done by a complete finite exclusion over the unique ordinary extremal skeleton, described in Section~\ref{sec:skeleton}.

\subsection{The unique skeleton for \(m=6,7\)}\label{sec:skeleton}

When \(m=7\) and \(|E_1|=18\), equality holds in both \eqref{eq:paircount} and \(\binom{d}{2}\ge 2d-3\). Hence every row has degree \(2\) or \(3\), with exactly four degree-\(3\) rows and three degree-\(2\) rows.

When \(m=6\) and \(|E_1|=16\), the sum of the differences \((d_i-2)(d_i-3)/2\) is at most \(1\). Only three candidate degree profiles are possible:
\begin{enumerate}
    \item four degree-\(3\) rows and two degree-\(2\) rows;
    \item one degree-\(1\) row and five degree-\(3\) rows;
    \item one degree-\(4\) row, two degree-\(3\) rows, and three degree-\(2\) rows.
\end{enumerate}
The second profile is impossible because each column can enter at most two degree-\(3\) rows, so five degree-\(3\) rows would require fifteen column incidences, while only twelve are available. The third profile is impossible because two degree-\(3\) rows and one degree-\(4\) row intersect in at most one column each; hence both degree-\(3\) rows must contain the two columns outside the degree-\(4\) row, creating a repeated column pair. Therefore only the first profile survives: four degree-\(3\) rows and two degree-\(2\) rows.

In both cases, the four degree-\(3\) rows make each of the six columns appear exactly twice. Viewing the four degree-\(3\) rows as four vertices, each column joins the two degree-\(3\) rows containing it, yielding a \(K_4\) without multiple edges. Hence the degree-\(3\) row system is unique. The unused column pairs are three mutually symmetric opposite pairs: for \(m=6\) one chooses two of them, and for \(m=7\) one chooses all three. Thus the ordinary extremal skeleton is unique up to relabelling rows and columns in each of the two dimensions.

\subsection{Complete finite exclusion and independent certification}\label{sec:certification}

Starting from the unique ordinary extremal skeleton of Section~\ref{sec:skeleton}, one enumerates all admissible pairings and empty-cell positions of the non-single-edge cells. Degenerate pairs in the same row, degenerate pairs in the same column, and complementary double edges are all allowed; no additional weak admissibility condition is imposed. The four necessary exclusion tasks are:

\begin{table}[htbp]
\centering
\begin{tabular}{c|cccc}
\toprule
\(m\) & target total edges & \(H\) & pruned backtracking states & result\\
\midrule
6 & 26 & 0 & 114 & no admissible configuration\\
6 & 25 & 2 & 4810 & no admissible configuration\\
7 & 30 & 0 & 1466 & no admissible configuration\\
7 & 29 & 2 & 115972 & no admissible configuration\\
\bottomrule
\end{tabular}
\caption{Finite exclusions for \(m=6,7\). The state counts are pruned backtracking states, not raw pairings, non-isomorphic graphs, or random trials.}
\label{tab:exclusions}
\end{table}

All four searches terminate normally and reproduce the same exclusion conclusion when rerun from the final data package. Larger values are already excluded by the grid upper bound, so together with the witnesses this gives
\[
z_{RL}(6,6)=24,\qquad z_{RL}(7,6)=28.
\]

\paragraph{Why pruning is valid.}
Deleting a complete selected double edge and its occupied cells preserves \((\mathrm{RW}3^+)\): replaying the retained derivation in the original certificate order, if a rectangle-adjacent diagonal encounters a deleted cell, that cell is now empty and can be certified as zero directly; otherwise the original pairing prescription is unchanged. When one edge of a complementary pair is deleted, the other edge is also initiated by an empty cell. Normalized saturation propagates only within complete edges. Therefore any subfamily of one, two, or three edges of an admissible completion, as well as any pairing prefix, must also be admissible. The program uses this necessary condition for pruning; the remaining pairing-graph isolated points, odd connected components, and untouched row permutations provide only additional necessary prunings and never delete an admissible completion.

\paragraph{Complete certification.}
All eleven grids are first checked for simplicity, \(C_4\)-freeness, and \(|E_1|=z(m,6)\); then the closure is generated according to the row/column, saturation, rectangle-transitivity, and complementary-pair rules. Finally it is required that every double edge is normalized, distinct selected edges remain in distinct classes, and every pair of distinct selected edges has an orthogonal certificate. An independent replay checks each rule line by line and does not depend on the search score. The number of uncertified edge pairs is zero for every grid.

A fast decision procedure and a cell-wise closure agree on 1907 random or structured samples, including all double-edge subfamilies of the two witnesses for \(m=6,7\). This is an additional cross-check and does not replace the complete enumeration for the upper bound. The data package contains all pairing tables, cell-wise certificates, search sources, and termination records; no floating-point tolerances, numerical Gram-matrix decisions, or odd-cycle reasoning from unknown starting points are used.

\subsection{Reproducibility commands}

Python uses only the standard library, runs normally, and does not use assertions-disabling \texttt{-O}:
\begin{verbatim}
python verify_all.py
python computation/rerun_upper.py
\end{verbatim}
The first command checks all witnesses and independently replays them; the second reruns the four exclusions with a \(\mathrm{C{+}{+}17}\) compiler supporting \texttt{unsigned\_\_int128}. Internal JSON cell numbering starts from zero; the figures number rows and columns starting from one.

\subsection{The figures}

\begin{figure}[htbp]
\centering
\begin{minipage}[t]{0.48\textwidth}\centering\scriptsize
$m=6:\ z_{RL}=24,\ |E_1|=16,\ |E_2|=8,\ H=4$\\
\[
\begin{array}{c|cccccc}
 & 1 & 2 & 3 & 4 & 5 & 6\\\hline
1 & \bullet & \bullet & 7 & 1 & 2 & 5\\
2 & \circ & 2 & \bullet & \bullet & 5 & 4\\
3 & \bullet & \circ & \bullet & 6 & \bullet & 3\\
4 & 7 & \bullet & 4 & \bullet & \bullet & \circ\\
5 & 6 & \bullet & \bullet & \circ & 8 & \bullet\\
6 & \bullet & 3 & 8 & \bullet & 1 & \bullet\\
\end{array}
\]
\end{minipage}
\hfill
\begin{minipage}[t]{0.48\textwidth}\centering\scriptsize
$m=7:\ z_{RL}=28,\ |E_1|=18,\ |E_2|=10,\ H=4$\\
\[
\begin{array}{c|cccccc}
 & 1 & 2 & 3 & 4 & 5 & 6\\\hline
1 & \bullet & \bullet & 1 & 2 & 6 & 3\\
2 & \circ & 9 & \bullet & \bullet & 5 & 2\\
3 & \bullet & \circ & \bullet & 1 & \bullet & 4\\
4 & 3 & \bullet & 8 & \bullet & \bullet & \circ\\
5 & 9 & \bullet & \bullet & 5 & 7 & \bullet\\
6 & \bullet & 6 & \circ & \bullet & 8 & \bullet\\
7 & 7 & 4 & 10 & 10 & \bullet & \bullet\\
\end{array}
\]
\end{minipage}
\\[1mm]
\begin{minipage}[t]{0.48\textwidth}\centering\scriptsize
$m=8:\ z_{RL}=33,\ |E_1|=19,\ |E_2|=14,\ H=1$\\
\[
\begin{array}{c|cccccc}
 & 1 & 2 & 3 & 4 & 5 & 6\\\hline
1 & \bullet & \bullet & 1 & 11 & 2 & 10\\
2 & \bullet & 7 & \bullet & 7 & 12 & 11\\
3 & 4 & \bullet & \bullet & 3 & 6 & 4\\
4 & \bullet & 9 & 8 & \bullet & 5 & \circ\\
5 & 13 & \bullet & 9 & 14 & \bullet & 8\\
6 & 6 & 10 & \bullet & \bullet & \bullet & 5\\
7 & 14 & \bullet & 2 & \bullet & 1 & \bullet\\
8 & \bullet & 12 & 13 & 3 & \bullet & \bullet\\
\end{array}
\]
\end{minipage}
\hfill
\begin{minipage}[t]{0.48\textwidth}\centering\scriptsize
$m=9:\ z_{RL}=37,\ |E_1|=21,\ |E_2|=16,\ H=1$\\
\[
\begin{array}{c|cccccc}
 & 1 & 2 & 3 & 4 & 5 & 6\\\hline
1 & \bullet & \bullet & 1 & \circ & 2 & 7\\
2 & \bullet & 5 & \bullet & 15 & 10 & 6\\
3 & 3 & \bullet & \bullet & 4 & 8 & 9\\
4 & \bullet & 6 & 14 & \bullet & 8 & 5\\
5 & 4 & \bullet & 16 & 3 & \bullet & 12\\
6 & 11 & 14 & \bullet & \bullet & \bullet & 13\\
7 & 9 & 10 & \bullet & 7 & 11 & \bullet\\
8 & 13 & \bullet & 2 & \bullet & 1 & \bullet\\
9 & \bullet & 16 & 15 & 12 & \bullet & \bullet\\
\end{array}
\]
\end{minipage}
\\[1mm]
\caption{Extremal grids for \(m=6,7,8,9\). Each \(\bullet\) is an independent one-edge; \(\circ\) is an empty cell. Two occurrences of the same positive integer form one double edge. Each figure is numbered independently; double edges occupy two cells and count once.}
\label{fig:6-9}
\end{figure}

\begin{figure}[htbp]
\centering
\begin{minipage}[t]{0.48\textwidth}\centering\scriptsize
$m=10:\ z_{RL}=41,\ |E_1|=22,\ |E_2|=19,\ H=0$\\
\[
\begin{array}{c|cccccc}
 & 1 & 2 & 3 & 4 & 5 & 6\\\hline
1 & \bullet & \bullet & 4 & 2 & 3 & 5\\
2 & \bullet & 17 & \bullet & 3 & 2 & 13\\
3 & \bullet & 19 & 5 & \bullet & 10 & 4\\
4 & 1 & \bullet & \bullet & \bullet & 1 & 18\\
5 & 12 & \bullet & 10 & 8 & \bullet & 9\\
6 & 13 & 16 & \bullet & 9 & \bullet & 8\\
7 & 6 & 18 & 7 & \bullet & \bullet & 17\\
8 & 7 & \bullet & 6 & 15 & 12 & \bullet\\
9 & 19 & 14 & \bullet & 11 & 14 & \bullet\\
10 & \bullet & 15 & 16 & 11 & \bullet & \bullet\\
\end{array}
\]
\end{minipage}
\hfill
\begin{minipage}[t]{0.48\textwidth}\centering\scriptsize
$m=11:\ z_{RL}=45,\ |E_1|=24,\ |E_2|=21,\ H=0$\\
\[
\begin{array}{c|cccccc}
 & 1 & 2 & 3 & 4 & 5 & 6\\\hline
1 & \bullet & \bullet & 3 & 16 & 2 & 3\\
2 & \bullet & 13 & \bullet & 12 & 15 & 12\\
3 & \bullet & 15 & 8 & \bullet & 9 & 17\\
4 & 1 & \bullet & \bullet & \bullet & 1 & 19\\
5 & 4 & \bullet & 19 & 5 & \bullet & 16\\
6 & 5 & 10 & \bullet & 4 & \bullet & 11\\
7 & 13 & 11 & 14 & \bullet & \bullet & 10\\
8 & 20 & \bullet & 9 & 18 & 8 & \bullet\\
9 & 21 & 17 & \bullet & 20 & 2 & \bullet\\
10 & 14 & 6 & 7 & \bullet & 18 & \bullet\\
11 & \bullet & 7 & 6 & 21 & \bullet & \bullet\\
\end{array}
\]
\end{minipage}
\\[1mm]
\begin{minipage}[t]{0.48\textwidth}\centering\scriptsize
$m=12:\ z_{RL}=48,\ |E_1|=25,\ |E_2|=23,\ H=1$\\
\[
\begin{array}{c|cccccc}
 & 1 & 2 & 3 & 4 & 5 & 6\\\hline
1 & \bullet & 21 & 22 & 2 & 18 & 23\\
2 & \bullet & \bullet & 1 & 8 & 4 & 5\\
3 & 10 & \bullet & \bullet & 10 & 16 & 9\\
4 & 5 & \bullet & 12 & \bullet & 7 & 14\\
5 & 6 & 11 & \bullet & \bullet & 14 & 7\\
6 & 12 & \bullet & 13 & 19 & \bullet & 8\\
7 & 16 & 20 & \bullet & 15 & \bullet & 21\\
8 & \bullet & 23 & 9 & \bullet & \bullet & 19\\
9 & 11 & \bullet & 18 & 15 & 1 & \bullet\\
10 & \bullet & 17 & \bullet & 3 & 22 & \bullet\\
11 & \circ & 13 & 6 & \bullet & 4 & \bullet\\
12 & 20 & 3 & 2 & 17 & \bullet & \bullet\\
\end{array}
\]
\end{minipage}
\hfill
\begin{minipage}[t]{0.48\textwidth}\centering\scriptsize
$m=13:\ z_{RL}=52,\ |E_1|=27,\ |E_2|=25,\ H=1$\\
\[
\begin{array}{c|cccccc}
 & 1 & 2 & 3 & 4 & 5 & 6\\\hline
1 & \bullet & \bullet & 14 & 19 & 20 & 10\\
2 & \bullet & 25 & \bullet & 3 & 7 & \circ\\
3 & 9 & \bullet & \bullet & 4 & 12 & 11\\
4 & \bullet & 6 & 8 & \bullet & 21 & 25\\
5 & 24 & \bullet & 18 & \bullet & 8 & 17\\
6 & 12 & 10 & \bullet & \bullet & 5 & 6\\
7 & \bullet & 16 & 16 & 13 & \bullet & 5\\
8 & 4 & \bullet & 22 & 24 & \bullet & 3\\
9 & 11 & 1 & \bullet & 19 & \bullet & 9\\
10 & \bullet & 23 & 13 & 15 & 2 & \bullet\\
11 & 17 & \bullet & 2 & 20 & 7 & \bullet\\
12 & 18 & 15 & \bullet & 23 & 1 & \bullet\\
13 & 14 & 22 & 21 & \bullet & \bullet & \bullet\\
\end{array}
\]
\end{minipage}
\\[1mm]
\caption{Extremal grids for \(m=10,11,12,13\). Conventions as in Figure~\ref{fig:6-9}.}
\label{fig:10-13}
\end{figure}

\begin{figure}[htbp]
\centering
\begin{minipage}[t]{0.48\textwidth}\centering\scriptsize
$m=14:\ z_{RL}=56,\ |E_1|=28,\ |E_2|=28,\ H=0$\\
\[
\begin{array}{c|cccccc}
 & 1 & 2 & 3 & 4 & 5 & 6\\\hline
1 & \bullet & 8 & 2 & 25 & 4 & 7\\
2 & \bullet & \bullet & 20 & 17 & 15 & 19\\
3 & \bullet & 22 & \bullet & 14 & 14 & 9\\
4 & 11 & \bullet & \bullet & 11 & 25 & 27\\
5 & \bullet & 24 & 5 & \bullet & 10 & 1\\
6 & 7 & \bullet & 10 & \bullet & 5 & 4\\
7 & 18 & 21 & \bullet & \bullet & 19 & 28\\
8 & 20 & \bullet & 12 & 26 & \bullet & 9\\
9 & 13 & 3 & \bullet & 6 & \bullet & 21\\
10 & 28 & 22 & 16 & \bullet & \bullet & 16\\
11 & 18 & \bullet & 1 & 12 & 27 & \bullet\\
12 & 26 & 13 & \bullet & 15 & 17 & \bullet\\
13 & 23 & 2 & 8 & \bullet & 24 & \bullet\\
14 & \bullet & 6 & 23 & 3 & \bullet & \bullet\\
\end{array}
\]
\end{minipage}
\hfill
\begin{minipage}[t]{0.48\textwidth}\centering\scriptsize
$m=15:\ z_{RL}=60,\ |E_1|=30,\ |E_2|=30,\ H=0$\\
\[
\begin{array}{c|cccccc}
 & 1 & 2 & 3 & 4 & 5 & 6\\\hline
1 & \bullet & \bullet & 8 & 25 & 1 & 13\\
2 & \bullet & 11 & \bullet & 27 & 11 & 14\\
3 & 15 & \bullet & \bullet & 16 & 26 & 19\\
4 & \bullet & 3 & 22 & \bullet & 25 & 18\\
5 & 5 & \bullet & 28 & \bullet & 24 & 1\\
6 & 21 & 14 & \bullet & \bullet & 19 & 26\\
7 & \bullet & 9 & 18 & 6 & \bullet & 27\\
8 & 16 & \bullet & 17 & 13 & \bullet & 15\\
9 & 3 & 10 & \bullet & 7 & \bullet & 8\\
10 & 2 & 20 & 10 & \bullet & \bullet & 20\\
11 & \bullet & 23 & 24 & 4 & 28 & \bullet\\
12 & 7 & \bullet & 22 & 6 & 23 & \bullet\\
13 & 12 & 9 & \bullet & 12 & 4 & \bullet\\
14 & 30 & 30 & 5 & \bullet & 21 & \bullet\\
15 & 17 & 2 & 29 & 29 & \bullet & \bullet\\
\end{array}
\]
\end{minipage}
\\[1mm]
\begin{minipage}[t]{0.48\textwidth}\centering\scriptsize
$m=16:\ z_{RL}=63,\ |E_1|=31,\ |E_2|=32,\ H=1$\\
\[
\begin{array}{c|cccccc}
 & 1 & 2 & 3 & 4 & 5 & 6\\\hline
1 & \bullet & 25 & 24 & 7 & 31 & 16\\
2 & \bullet & \bullet & 28 & 22 & 14 & 6\\
3 & \bullet & 26 & \bullet & 27 & 24 & 11\\
4 & 2 & \bullet & \bullet & 10 & 17 & 2\\
5 & \bullet & 29 & 31 & \bullet & 28 & 12\\
6 & 18 & \bullet & 14 & \bullet & 21 & 9\\
7 & 17 & 19 & \bullet & \bullet & 3 & 6\\
8 & \bullet & 23 & 12 & 19 & \bullet & 23\\
9 & 30 & \bullet & 7 & 27 & \bullet & 1\\
10 & 3 & 18 & \bullet & 8 & \bullet & 32\\
11 & 9 & 5 & 1 & \bullet & \bullet & 30\\
12 & \bullet & 11 & \circ & 4 & 21 & \bullet\\
13 & 32 & \bullet & 29 & 16 & 15 & \bullet\\
14 & 20 & 25 & \bullet & 5 & 13 & \bullet\\
15 & 8 & 13 & 4 & \bullet & 26 & \bullet\\
16 & 22 & 10 & 20 & 15 & \bullet & \bullet\\
\end{array}
\]
\end{minipage}
\\[1mm]
\caption{Extremal grids for \(m=14,15,16\). Conventions as in Figure~\ref{fig:6-9}.}
\label{fig:14-16}
\end{figure}

For \(m=6\):
\[
|E_1|=16,\quad |E_2|=8,\quad H=4,
\]
and all \(276\) distinct edge pairs pass certification.

For \(m=7\):
\[
|E_1|=18,\quad |E_2|=10,\quad H=4,
\]
and all \(378\) distinct edge pairs pass certification.

For \(m=8\):
\[
|E_1|=19,\quad |E_2|=14,\quad H=1,
\]
and all \(528\) distinct edge pairs pass certification.

For \(m=9\):
\[
|E_1|=21,\quad |E_2|=16,\quad H=1,
\]
and all \(666\) distinct edge pairs pass certification.

For \(m=10\):
\[
|E_1|=22,\quad |E_2|=19,\quad H=0,
\]
and all \(820\) distinct edge pairs pass certification.

For \(m=11\):
\[
|E_1|=24,\quad |E_2|=21,\quad H=0,
\]
and all \(990\) distinct edge pairs pass certification.

For \(m=12\):
\[
|E_1|=25,\quad |E_2|=23,\quad H=1,
\]
and all \(1128\) distinct edge pairs pass certification.

For \(m=13\):
\[
|E_1|=27,\quad |E_2|=25,\quad H=1,
\]
and all \(1326\) distinct edge pairs pass certification.

For \(m=14\):
\[
|E_1|=28,\quad |E_2|=28,\quad H=0,
\]
and all \(1540\) distinct edge pairs pass certification.

For \(m=15\):
\[
|E_1|=30,\quad |E_2|=30,\quad H=0,
\]
and all \(1770\) distinct edge pairs pass certification.

For \(m=16\):
\[
|E_1|=31,\quad |E_2|=32,\quad H=1,
\]
and all \(1953\) distinct edge pairs pass certification.

\subsection{Parity relation for \(m\ge 15\)}

For \(m\ge 15\), the one-edge count is
\[
|E_1|=m+15.
\]
From \eqref{eq:unified6} and \(|E_1|=m+15\),
\[
N=|E_1|+|E_2|=\frac{7m+15-H}{2},
\]
so
\[
H\equiv m-1\pmod{2}.
\]
Thus, when the counting upper bound is attained, odd rows should have no empty cell, and even rows should have exactly one empty cell. The implementations for \(m=15,16\) realize this; the present finite computation does not prove that all larger six-column sizes attain the upper bound.

\section{Conclusions}\label{sec:conclusions}

We have proved that
\[
z_{RL}(15,6)=60,
\]
thereby eliminating the original motivation for introducing the signed Zarankiewicz number \(z_{SL}\) in the exceptional case \(m=15\), \(n=6\). We have also determined the exact values of \(z_{RL}(m,6)\) for all \(6\le m\le 16\):
\[
\begin{array}{c|ccccccccccc}
m & 6 & 7 & 8 & 9 & 10 & 11 & 12 & 13 & 14 & 15 & 16\\\hline
z_{RL}(m,6) & 24 & 28 & 33 & 37 & 41 & 45 & 48 & 52 & 56 & 60 & 63
\end{array}
\]
For \(m=8,\dots,16\), the constructions attain the grid counting upper bound, so optimality follows directly. For \(m=6,7\), the unique ordinary extremal skeleton is identified and a complete finite exclusion rules out all larger total edge counts.

The general question whether \(z_{SL}(m,n)=z_{RL}(m,n)\) for all \(m,n\) remains open. In view of Lebedev's recent separation of \(z_A\) and \(z_L\) at \(m=n=1893\) \cite{leb26}, it is natural to ask whether an analogous separation exists for \(z_{SL}\) and \(z_{RL}\).

\section{Open Problems}\label{sec:open}

\begin{enumerate}
    \item \textbf{Equality of \(z_{SL}\) and \(z_{RL}\).} Is it true that
    \[
    z_{SL}(m,n)=z_{RL}(m,n)
    \]
    for all \(m,n\)? The case \(m=15\), \(n=6\) no longer provides a counterexample, but no general proof is known.

    \item \textbf{Separation of \(z_{SL}\) and \(z_{RL}\).} Does there exist a pair \((m,n)\) such that
    \[
    z_{SL}(m,n)>z_{RL}(m,n)?
    \]
    If so, find an explicit construction, in analogy with Lebedev's separation of \(z_A\) and \(z_L\) at \(m=n=1893\) \cite{leb26}.

    \item \textbf{Small-side six-column equality.} The equality
    \(z_2=z_{SL}=z_{RL}\) is known for \(5\times4\) and \(6\times4\)
    \cite{cc26mx4}, and for \(4\times4\), \(7\times4\), \(8\times4\),
    and \(5\times5\) \cite{xu26}. By
    Corollary~\ref{cor:second-order-six-columns}, it is also known for
    \(m\times6\) with \(8\le m\le16\). Do we have
    \[
    z_2(6,6)=z_{SL}(6,6)=z_{RL}(6,6)
    \quad\text{and}\quad
    z_2(7,6)=z_{SL}(7,6)=z_{RL}(7,6)?
    \]

    \item \textbf{Six columns beyond \(m=16\).} Determine
    \(z_{RL}(m,6)\) for \(m\ge17\). In particular, is
    \[
    z_{RL}(m,6)=\left\lfloor\frac{7m+15}{2}\right\rfloor
    \]
    for every \(m\ge15\)?

    \item \textbf{Four columns.} Exact small-parameter equalities for four
    columns are now known for \(m=4,5,6,7,8\)
    \cite{cc26mx4,xu26}. Can these methods determine further exact values
    or a general formula in that family?

    \item \textbf{Higher second-order parameters.} The chain \eqref{eq:chain} relates \(z_2\), \(z_{SL}\), and \(z_{RL}\). Are there further refinements of the Zarankiewicz number that are strictly stronger than \(z_{RL}\) and strictly weaker than \(z_2\)?
\end{enumerate}

\clearpage
\section{Data and computational reproducibility}
\label{sec:reproducibility}

The data package accompanying this paper contains:
\begin{itemize}
    \item all eleven extremal grids for \(m=6,\dots,16\), in machine-readable form;
    \item the cell-wise certificates for every double edge and every pair of distinct selected edges;
    \item the complete finite-exclusion search sources and termination records for \(m=6,7\);
    \item the independent replay tool \texttt{verify\_all.py};
    \item the exclusion rerun tool \texttt{computation/rerun\_upper.py}.
\end{itemize}
The verification commands are
\begin{verbatim}
python verify_all.py
python computation/rerun_upper.py
\end{verbatim}
The first command checks all witnesses and independently replays them. The second reruns the four exclusions with a \(\mathrm{C{+}{+}17}\) compiler supporting \texttt{unsigned\_\_int128}. Internal JSON cell numbering starts from zero; the figures number rows and columns starting from one. No floating-point tolerances, numerical Gram-matrix decisions, or odd-cycle reasoning from unknown starting points are used.

\section*{Acknowledgments}

This work was partially supported by Jiangsu Provincial Scientific Research Center of Applied Mathematics (Grant No.~BK20233002), Research Center for Intelligent Operations Research, The Hong Kong Polytechnic University (4-ZZT8), and the National Natural Science Foundation of China (Nos.~12171168, 12071159).


\begin{thebibliography}{99}
\bibitem{bo04} Bollob\'as, B. \emph{Extremal Graph Theory}. Dover Publications, 2004.

\bibitem{bryant} Bryant, D. et al. Perfect 1-factorizations. \emph{Mathematica Slovaca} \textbf{2019}, \emph{69}, 479--496.

\bibitem{chm24} Chen, G.; Horsley, D.; Mammoliti, A. Zarankiewicz numbers near the triple system threshold. \emph{J. Comb. Des.} \textbf{2024}, \emph{32}, 556--576.

\bibitem{cc26mx4} Chen, Z.; Chen, Y. Recursive-line Zarankiewicz numbers with four columns. \emph{arXiv} \textbf{2026}, arXiv:2609.11093v2.

\bibitem{clz95} Choi, M.~D.; Lam, T.~Y.; Reznick, B. Sums of squares of real polynomials. \emph{Proc. Symp. Pure Math.} \textbf{1995}, \emph{58}, 103--126.

\bibitem{culik56} \v{C}ul\'{\i}k, K. Teilweise L\"{o}sung eines verallgemeinerten Problems von K.~Zarankiewicz. \emph{Annales Polonici Mathematici} \textbf{1956}, \emph{3}, 165--168.

\bibitem{gu69} Guy, R.~K. A many-faceted problem of Zarankiewicz. In \emph{The Many Facets of Graph Theory}; Chartrand, G., Kapoor, S.~F., Eds.; Springer, 1969; pp.~129--141.

\bibitem{kobayashi1989} Kobayashi, M. On perfect one-factorization of the complete graph \(K_{2p}\). \emph{Graphs and Combinatorics} \textbf{1989}, \emph{5}, 351--353.

\bibitem{kst54} Kov\'ari, T.; S\'os, V.; Tur\'an, P. On a problem of K. Zarankiewicz. \emph{Coll. Math.} \textbf{1954}, \emph{3}, 50--57.

\bibitem{leb26} Lebedev, N. Density and separation for augmented Zarankiewicz numbers. \emph{arXiv} \textbf{2026}, arXiv:2609.16555v1.

\bibitem{reproducibility} L\"{o}fberg, J.; Qi, L. Second order Zarankiewicz number. \emph{arXiv} \textbf{2026}, arXiv:2608.30555v3.

\bibitem{ni10} Nikiforov, V. A contribution to the Zarankiewicz problem. \emph{Linear Algebra Appl.} \textbf{2010}, \emph{432}, 1405--1411.

\bibitem{qi1} Qi, L.; Cui, C.; Xu, Y. Biquadratic SOS rank and augmented Zarankiewicz number. \emph{Mathematics} \textbf{2026}, \emph{14}, 1552.

\bibitem{re58} Reiman, I. \"Uber ein Problem von K. Zarankiewicz. \emph{Acta Math. Acad. Sci. Hung.} \textbf{1958}, \emph{9}, 269--273.

\bibitem{xu26} Xu, Y.; Yan, X. Exact values, extremal classifications, and sum-of-squares reductions for second-order Zarankiewicz numbers. \emph{arXiv} \textbf{2026}, arXiv:2609.18429v1.

\bibitem{za51} Zarankiewicz, K. Problem P 101. \emph{Coll. Math.} \textbf{1951}, \emph{2}, 301.

\end{thebibliography}
\end{document}